\documentclass[12pt]{amsart}

\usepackage{times}
\usepackage{amsthm}
\usepackage{amssymb}
\usepackage{amsmath}
\usepackage{amsfonts} 
\usepackage{amsthm}
\usepackage{bm}
\usepackage{color}
\usepackage{graphicx}

\newcommand{\excise}[1]{}

\newtheorem{theorem}{Theorem}[section]
\newtheorem{lemma}[theorem]{Lemma}

\newtheorem{corollary}[theorem]{Corollary}
\newtheorem{proposition}[theorem]{Proposition}

\theoremstyle{definition}

\newtheorem{remark}[theorem]{Remark}

\newtheorem{definition}[theorem]{Definition}

        {\begin{list}
                {\noindent\makebox[0mm][r]{\arabic{enumi}.}}
                {\leftmargin=5.5ex \usecounter{enumi}}
        }
        {\end{list}}

        {\begin{list}
                {\noindent\makebox[0mm][r]{(\roman{enumi})}}
                {\leftmargin=5.5ex \usecounter{enumi}}
        }
        {\end{list}}

\def\<{\langle}
\def\>{\rangle}
\def\0{\mathbf{0}}
\def\CC{{\mathbb C}}

\def\NN{{\mathbb N}}

\def\QQ{{\mathbb Q}}
\def\RR{{\mathbb R}}

\def\ZZ{{\mathbb Z}}

\def\cB{{\mathcal B}}

\def\oJ{{\hspace{.45ex}\overline{\hspace{-.45ex}J}}}

\def\omu{\overline{\mu}}
\def\onu{\overline{\nu}}
\def\oalpha{\overline{\alpha}}
\def\ogamma{\overline{\gamma}}
\def\osigma{\overline{\sigma}}
\def\otau{\overline{\tau}}

\def\del{\partial}

\def\sat{{\rm sat}}

\def\cocoa{{\hbox{\rm C\kern-.13em o\kern-.07em C\kern-.13em o\kern-.15em A}}}

\numberwithin{equation}{section}
\begin{document}

\mbox{}
\vspace{-3ex}
\title{Weyl Closure of Hypergeometric Systems}

\author{Laura Felicia Matusevich}
\address{Department of Mathematics \\
Texas A\&M University \\ College Station, TX 77843.}
\email{laura@math.tamu.edu}
\thanks{The author was partially supported by NSF Grant DMS 0703866}

\subjclass[2000]{Primary: 33C70, 32C38; Secondary: 14M25, 13N10}

\begin{abstract}
We show that $A$-hypergeometric systems and Horn
hypergeometric systems are Weyl closed for
very generic parameters. 
\end{abstract}
\maketitle

\mbox{}
\vspace{-5.05ex}
\parskip=0ex
\parindent2em
\setcounter{tocdepth}{2}
\parskip=1ex
\parindent0pt

\section{Introduction}

Let $D=D_n$ be the (complex) Weyl algebra, that is, the ring of
linear partial differential operators with polynomial coefficients
in variables $x_1,\dots,x_n$ and $\del_1,\dots,\del_n$, where
$\del_i$ stands for $\frac{\del}{\del x_i}$. Write
$R = \CC(x)\otimes_{\CC[x]}D$ for the ring of operators with rational function
coefficients. If $I$ is a left $D$-ideal, then the \emph{Weyl closure} of $I$
is
\[
RI \cap D.
\]
If $I$ equals its Weyl closure, then $I$ is said to be \emph{Weyl closed}.

The operation of Weyl closure is an analog of the radical operation in
the polynomial ring, as the Weyl closure of $I$ is the differential annihilator
of the space 
of germs of holomorphic solutions of $I$ at a generic nonsingular point
(see Proposition 2.19 in \cite{tsai-thesis}).
The notion of Weyl closure was introduced by Harrison Tsai in \cite{tsai-thesis}.
This work contains an algorithm to compute the Weyl closure of a left
$D$-ideal, which has been implemented by Anton Leykin and Harrison Tsai
in the computer algebra system Macaulay2 \cite{M2}. Other references are
\cite{weyl-closure-1-var,weyl-closure-multivar}.

The goal of this note is to show that $A$-hypergeometric systems and
Horn hypergeometric systems are Weyl closed when the parameters are generic enough.
Our main result, Theorem~\ref{thm:A-hyp-weyl-closed}, gives a stronger
property than Weyl closure
for any $A$-hypergeometric system with very generic parameters: 
such a system is the differential annihilator of a \emph{single function}.
This has practical consequences: often we are interested in one specific
hypergeometric series $F$, which we would like to study through the differential
operators it satisfies. If the series in question is
a function of $m$ variables, traditional methods will provide $m$ differential
equations to form a Horn system that annihilates $F$, but in general, this
system will be strictly contained in the differential annihilator. A more
modern approach produces an $A$-hypergeometric system for which
$F$ is a solution. Theorem~\ref{thm:A-hyp-weyl-closed} says that
any other differential equation
$F$ satisfies will be a consequence of these $A$-hypergeometric ones.
There is an analogy to algebraic numbers: if one studies a finite extension
$\QQ(\lambda)$ of $\QQ$, then having a polynomial with rational coefficients
whose root is $\lambda$ is
useful, but what one really wants is the \emph{minimal} such polynomial.

We illustrate the typical situation in an example. 
Given $a,a' \in \CC \backslash \ZZ$, consider the series
\[
G(s,t) = \sum_{(m,n) \in \NN^2} c_{m,n}s^mt^n= 
\sum_{(m,n) \in \NN^2} \frac{(a)_{m-2n}}{(a')_{n-2m}} \frac{s^mt^n}{m!n!},
\]
where the Pochhammer symbol $(a)_k$ is given by
\[
(a)_k = 
\left\{
    \begin{array}{lr}
      \prod_{l=0}^{k-1} (a+l)                 & k \geq 0 \\
           &  \\
      ({\prod_{l=1}^{|k|} (a-l)})^{-1} \quad   & k < 0
    \end{array}
\right.
\]

This series converges in a neighborhood of the origin, and
it is a hypergeometric series, since its coefficients
satisfy the following special recurrence relations:
\[
\frac{c_{m+1,n}}{c_{m,n}} = \frac{(-2m+n+a'-1)(-2m+n+a'-2)}{(m+1)(m-2n+a)}
\]
\[  
\frac{c_{m,n+1}}{c_{m,n}} = \frac{(m-2n+a-1)(m-2n+a-2)}{(n+1)(-2m+n+a')}
\]
which translate into the following system of differential equations for $G$:
\begin{equation}
\label{eqn:horn1}
\begin{array}{r}
\big[ \frac{1}{s} \theta_s(\theta_s-2\theta_t+a-1) - 
(-2\theta_s +\theta_t +a'-1)(-2\theta_s+\theta_t+a'-2)\big] G(s,t) = 0 \\ \\
\big[ \frac{1}{t} \theta_t(-2\theta_s+\theta_t+a'-1) -
(\theta_s-2\theta_t+a-1)(\theta_s-2\theta_t+a-2) \big]G(s,t) = 0
\end{array}
\end{equation}
where $\theta_s = s\frac{\del}{\del s}$ and $\theta_t = t \frac{\del}{\del t}$.

{\bf Question:} Is the above system the \emph{differential annihilator}
of $G$? This would mean that 
any differential equation for $G$ can be obtained
by taking combinations (with coefficients in the Weyl algebra in $s,t$,
$\frac{\del}{\del s}, \frac{\del}{\del t}$) of the
equations (\ref{eqn:horn1}).

It turns out that it is simpler to study $G$, and in particular, determine
its differential annihilator, 
if we make a change of variables, as follows.

Define matrices
\[
B = \left[
\begin{array}{rr}
1 & 0 \\
-2 & 1 \\
1 & -2 \\
0 & 1 
\end{array}
\right]
;\quad
A=
\left[
\begin{array}{rrrr}
3 & 2 & 1 & 0 \\
0 & 1 & 2 & 3
\end{array}
\right].
\]
The rows of $B$ tell us the factors that appear in the differential equations
(\ref{eqn:horn1}); these factors are obtained by adding appropriate
parameters to dot products of rows of $B$ with the vector $(\theta_s,\theta_t)$.
The matrix $A$ is chosen so that the columns of $B$ form a basis for its kernel.

Let 
\[
F(x_1,x_2,x_3,x_4) = x_2^{a'-1}x_3^{a-1} G(\frac{x_1x_3}{x_2^2},\frac{x_2x_4}{x_3^2})
= x_2^{a'-1}x_3^{a-1} G(x^B).
\]

Then $F$ satisfies the following system of differential equations
\[
[3\theta_1 + 2\theta_2 + \theta_3 - (2a'+ a -3)] F(x) = 0 ; \quad
[\theta_2 + 2\theta_3 + 3\theta_4 - (2a + a'-3)] F(x) = 0;
\]
\[
[\del_1\del_3 - \del_2^2]F(x) = 0; \quad
[\del_2 \del_4 - \del_3^2] F(x) = 0,
\]
where $\del_i$ stands for $\frac{\del}{\del x_i}$ and $\theta_i = x_i \del_i$.
The first two differential equations reflect the change of variables we applied to $G$.
The last two correspond to the differential equations (\ref{eqn:horn1}).
We call this system a \emph{Horn system} (Definition~\ref{def:horn}).

It turns out (by Theorem~\ref{thm:A-hyp-weyl-closed})
that, in order to get the differential annihilator of $F$, we need to
add another equation, namely
\[
[\del_1 \del_4-\del_2\del_3]F(x) = 0.
\]
When we do this, we obtain an \emph{A-hypergeometric system} (Definition~\ref{def:A-hyp}).

The $A$-hypergeometric system is strictly larger than the Horn system. To see this,
note that, as $a, a' \in \CC\backslash \ZZ$,
the Puiseux monomial $x_1^{\frac{2a'+a-3}{3}}x_4^{\frac{2a+a'-3}{3}}$ does not equal
$1$. This monomial 
is a solution of the Horn system,
but not of the $A$-hypergeometric system. 
On the other hand, by
Corollary~\ref{coro:horn-weyl-closed}, the Horn system itself is also Weyl closed
when the parameters are very generic.

This is interesting information. It tells us, for instance, that the function
$x_1^{\frac{2a'+a-3}{3}}x_4^{\frac{2a+a'-3}{3}}$ cannot be obtained from $F$ by analytic continuation,
a fact that was already known to Erd\'elyi \cite{erdelyi}, although he could not
justify it. From our perspective, the reason is simple: any function obtained from
$F$ by analytic continuation must satisfy the same differential equations as $F$,
i.e. it has to be a solution of the differential annihilator of $F$.
Thus, if we want
to understand the monodromy of the function $F$, the differential equations
we should study are the $A$-hypergeometric system, and not the smaller Horn system.

The plan for this article is as follows. In Section~\ref{sec:a-hyp}, we 
define $A$-hypergeometric systems, and show that they are Weyl closed for 
very generic parameters (Corollary~\ref{coro:A-hyp-weyl-closed}).
A key ingredient is the existence of \emph{fully supported}
(Definition~\ref{def:fully-supported}) convergent power series solutions 
of $A$-hypergeometric systems \cite{GKZ,SST,hypergeometric-series-solutions}.
In Section~\ref{sec:horn} we introduce Horn systems, and again, prove that they 
are Weyl closed for very generic parameters. The proofs in this section
rely heavily on results from \cite{dmm}.

\subsection*{Acknowledgements}

I am very grateful to Harry Tsai, for interesting conversations on the
subject of Weyl closure, as well as to Mutsumi Saito and Shintaro Kusumoto, who
found a mistake that has now been corrected.
Thanks also to Alicia Dickenstein, Bernd
Sturmfels and Ezra Miller, 
who made helpful comments on an earlier version of this 
manuscript. 
I especially thank
the referee, whose thoughtful suggestions have improved
this article.

\section{A-hypergeometric Systems}
\label{sec:a-hyp}

We will work in the Weyl algebra $D=D_n$ in $x_1,\dots,x_n$, $\del_1,\dots,\del_n$, 
and denote $\theta_j = x_j \del_j$.

Let $A = (a_{ij})$ be a $d\times n$ integer matrix of full rank $d$, satisfying two
conditions on its columns. The first is that they $\ZZ$-span $\ZZ^d$, and the
second is that they all lie in an open half space of $\RR^d$. In particular,
$A$ is not allowed to have a zero column.

\begin{definition}
\label{def:A-hyp}
Given $A$ as above, set
\[
E_i = \sum_{j=1}^n a_{ij} \theta_j : \quad i = 1,\dots,d,
\]
and define the \emph{toric ideal} to be
\[
I_A = \langle \del^u -\del^v : u,v \in \NN^n, Au = Av \rangle \subseteq \CC[\del_1,\dots,\del_n].
\]
For $\beta \in \CC^d$ the \emph{$A$-hypergeometric system with parameter $\beta$}
is the left $D$-ideal
\[
H_A(\beta) = I_A + \langle E-\beta \rangle \subseteq D,
\]
where $\langle E-\beta \rangle$ is shorthand for 
$\langle E_i - \beta_i : i = 1,\dots,d \rangle$.
\end{definition}

Note that although $\langle E-\beta \rangle$ depends on the matrix $A$,
this is not reflected in the notation.

$A$-hypergeometric systems were introduced in the work of Gelfand, Graev,
Kapranov and Zelevinsky \cite{GGZ,GKZ}.
The text \cite{SST} emphasizes computational aspects in the theory of
$A$-hypergeometric equations, and is highly recommended.

We wish to show that $H_A(\beta)$ is Weyl closed for very generic $\beta$.
Here, \emph{very generic} will mean 
``outside a countable locally finite collection of 
algebraic varieties''. 
The following definition gives us a
countable family of the hyperplanes that we will need to
avoid.

\begin{definition}
\label{def:nonresonant}
A \emph{facet} of $A$ is a subset of its columns that is maximal
among those minimizing nonzero linear functionals on~$\ZZ^d$.  
Denote the columns of $A$ by $a_1,\dots,a_n$.
Geometrically, the facets of $A$ correspond to the facets of the cone
$\RR_{\geq 0} A = \{ \sum_{i=1}^n \lambda_i a_i : \lambda_i \in \RR_{\geq 0}\} \subseteq \RR^d$,
all of which contain the origin.
For a facet $\sigma$ of~$A$ let $\nu_\sigma$ be its \emph{primitive
support function}, the unique rational linear form satisfying
\begin{enumerate}
\item $\nu_\sigma(\ZZ A) = \ZZ$,
\item $\nu_\sigma(a_j) \geq 0$ for all $j \in \{1,\dots,n\}$,
\item $\nu_\sigma(a_j) = 0$ for all $a_j \in \sigma$.
\end{enumerate}
A parameter vector $\beta \in \CC^d$ is \emph{$A$-nonresonant}\/ 
(or simply \emph{nonresonant}, when it causes no confusion) if
$\nu_\sigma(\beta) \notin \ZZ$ for all facets $\sigma$ of~$A$.  Note
that if $\beta$ is $A$-nonresonant, then so is $\beta+A\gamma$
for any $\gamma \in \ZZ^n$.
\end{definition}

Nonresonant parameters have nice properties, as is illustrated below.

\begin{lemma}
\label{lemma:deli}
Fix a nonresonant parameter $\beta$.
If $P \del_i \in H_A(\beta)$, then $P \in H_A(\beta - Ae_i)$.
\end{lemma}

This is an immediate consequence of the following well known fact, a concise
proof of which can be found in \cite{dmm}[Lemma 7.10].

\begin{theorem}
\label{thm:iso}
If $\beta$ is nonresonant, the map $D/H_A(\beta) \rightarrow D/H_A(\beta+Ae_i)$
given by right multiplication by $\del_i$ is an isomorphism.
\end{theorem}

We want to show that an $A$-hypergeometric system is the differential annihilator
of a special kind of function, that we define below.

\begin{definition}
\label{def:fully-supported}
A formal power series $\varphi \in \CC[[x_1,\dots,x_n]]$ 
is \emph{supported on a translate of a 
lattice $L \subseteq \ZZ^n$}
if it is of the form $x^v \sum_{u \in L} \lambda_u x^u$.
The set $\{ v+u : \lambda_u \neq 0 \}$ is called the \emph{support} of $\varphi$.
If the support of $\varphi$ is Zariski dense in the Zariski closure of
$v+L$, then $\varphi$ is \emph{fully supported}.
\end{definition}

We can guarantee the existence of fully supported solutions
of $H_A(\beta)$ if we require that the parameters be generic.

\begin{theorem}
\label{thm:good-sol}
If $\beta$ is generic, then $H_A(\beta)$ has a holomorphic solution that
can be represented as a fully supported power series on a
translate of the lattice
$\ker_{\ZZ}(A)$.
\end{theorem}

\begin{proof}
This follows from \cite{SST}[Proposition 3.4.4, Lemma 3.4.6] in the
case that the toric ideal $I_A$ is homogeneous in the usual grading of the 
polynomial ring $\CC[\del]$. Another proof can be found in \cite{GKZ}.
When $I_A$ is not homogeneous, we use
\cite{hypergeometric-series-solutions}[Theorem 2].
\end{proof}

We are now ready to prove the main result in this section.

\begin{theorem}
\label{thm:A-hyp-weyl-closed}
If $P \in D$ annihilates a fully supported solution $f$ of $H_A(\beta)$,
and $\beta$ is nonresonant, then $P \in H_A(\beta)$.
\end{theorem}

\begin{proof}
Let $f$ 
be a fully supported solution of $H_A(\beta)$, and let
$P \in D$ such that $P f = 0$.

The Weyl algebra is $\ZZ^d$-graded via $\deg(x^{\mu}\del^{\nu}) = A(\nu-\mu)$.
If $x^{\mu}\del^{\nu}$ and $x^{\mu'}\del^{\nu'}$ have different $A$-degrees, then
$x^{\mu}\del^{\nu} f$ and $x^{\mu'}\del^{\nu'} f$ have disjoint
supports, since $f$ is supported on a translate of the lattice $\ker_{\ZZ}(A)$. 
Thus, we may assume that $P$ is an $A$-homogeneous differential operator.

Moreover, if $x^{\mu^o}\del^{\nu^o}$ is a monomial with nonzero coefficient
in $P$, then $P\del^{\mu^o}$ is homogeneous of degree 
$A\nu^o \in \NN A \subseteq \ZZ^d$.
By Theorem~\ref{thm:iso}, we can find a solution $g$ of
$H_A(\beta+A\mu^o)$ such that $\del^{\mu^o}g=f$. 
Since $f$ is fully supported,
so is $g$. Finally $\beta+A\mu^o$ is nonresonant.

Write 
$P\del^{\mu^o}=\sum_{\mu,\nu} c_{\mu,\nu} x^{\mu}\del^{\nu}$. 
Since this operator is $A$-homogeneous of degree $A \nu^o$, we have
$A(\nu-\mu) = A\nu^o$, or equivelently, $A\nu = A(\mu+\nu^o)$, for
all $\mu, \nu$ such that $c_{\mu,\nu} \neq 0$.

Now 
\[
P\del^{\mu^o}
= \sum c_{\mu,\nu} x^{\mu} \del^{\nu} 
=  \sum c_{\mu,\nu} x^{\mu} (\del^{\nu} - \del^{\mu+\nu^o}) + \sum c_{\mu,\nu} x^{\mu} \del^{\mu}\del^{\nu^o}.
\]

Note that the binomial $\del^{\nu} - \del^{\mu+\nu^o} \in I_A \subseteq H_A(\beta+A\mu^o)$.
Then $(\sum c_{\mu,\nu} x^{\mu} \del^{\mu}) \del^{\nu^o}$ annihilates $g$, so
$\sum c_{\mu,\nu} x^{\mu} \del^{\mu}$ annihilates $\del^{\nu^o}g$, which is a 
solution of $H_A(\beta+A\mu^o-A\nu^o)$.

We claim that $\del^{\nu^o}g$ is fully supported. As $\beta + A\mu^o - A\nu^o$ 
is nonresonant, right multiplication by $\del^{\nu^o}$ is an isomorphism 
between $D/H_A(\beta+A\mu^o-A\nu^o)$ and $D/H_A(\beta+A\mu^o)$, whose inverse is 
a differential operator we denote $\del^{-\nu^o}$. Then $g = \del^{-\nu^o}(\del^{\nu^o} g)$
is fully supported, and therefore $\del^{\nu^o}g$ is as well.

As
\[
x^{\mu}\del^{\mu} = \prod_{j=1}^n \prod_{k=0}^{\mu_j-1}(\theta_j-k),
\]
we can write  $\sum c_{\mu,\nu} x^{\mu} \del^{\mu} = p(\theta_1,\dots,\theta_n)$ for
some polynomial $p$. Write 
\[
\del^{\nu^o}g = x^{v} \sum_{u \in \ker_{\ZZ}(A)} \lambda_u x^u,
\]
where $Av = \beta+A\mu^o-A\nu^o$.
Then
\[
0=\left[ \sum c_{\mu,\nu} x^{\mu} \del^{\mu} \right] (\del^{\nu^o}g)
= \left[ p(\theta_1,\dots,\theta_n) \right] (\del^{\nu^o}g) =
\sum_{u \in \ker_{\ZZ}(A)} \lambda_u p(v+u) x^{v+u},
\]
so that $p(v+u) = 0$ whenever $\lambda_u\neq 0$. But the 
fact that $g$ is fully supported means that the set
$\{v+u : \lambda_u \neq 0 \}$ is Zariski-dense in $v + \ker(A)$, so $p$ must
vanish on all of $v+\ker(A)$.
By the Nullstellensatz, this implies that 
$p(\theta_1,\dots,\theta_n)=\sum c_{\mu,\nu} x^{\mu}\del^{\mu}$ belongs to
$\langle E-(\beta + A\mu^o - A\nu^o) \rangle \subseteq 
H_A(\beta + A\mu^o-A\nu^o)$,
and so
$(\sum c_{\mu,\nu} x^{\mu}\del^{\mu}) \del^{\nu^o} \in H_A(\beta + A\mu^o)$.
But then $P\del^{\mu^o} \in H_A(\beta + A\mu^o)$, and using Lemma~\ref{lemma:deli},
we obtain $P \in H_A(\beta)$.
\end{proof}

\begin{corollary}
\label{coro:A-hyp-weyl-closed}
If $\beta$ is very generic, then $H_A(\beta)$ is Weyl closed.
\end{corollary}

\begin{proof}
If we choose $\beta$ generic so that $H_A(\beta)$ has a fully supported
series solution and also require $\beta$ to be nonresonant, we
fall into the hypotheses of Theorem~\ref{thm:A-hyp-weyl-closed},
which implies that $H_A(\beta)$ is a differential annihilator, and therefore
Weyl closed.
\end{proof}

\section{Horn Systems}
\label{sec:horn}

In this section, we show that Horn systems are Weyl closed for very generic parameters.

Let $B$ be an $n \times m$ integer matrix of full rank $m$ such that every nonzero
element of the $\ZZ$-column span of $B$ is \emph{mixed}, meaning that each such vector
has a strictly positive and a strictly negative entry. In particular, the columns of
$B$ are mixed. In this case, we can find a matrix $A$ as in Section~\ref{sec:a-hyp}
with $d=n-m$ such that $A B = 0$.

\begin{definition}
\label{def:horn}
Let $B$ and $A$ be matrices as above.
Given $u \in \ZZ^n$, write $u_+$ for the vector defined by
$(u_+)_i = u_i$ if $u_i \geq 0$, and $(u_+)_i=0$ otherwise. Let
$u_- = u_+-u$. 
The ideal
\[
I(B) = 
\langle \del^{u_+} - \del^{u_-} : u \;\mbox{is a column of} \;B \rangle \subseteq \CC[\del]
\]
is called a \emph{lattice basis ideal} for the lattice $\ZZ B$ spanned by the columns of $B$.
For any $\beta \in \CC^d$ the left $D$-ideal
\[
H(B,\beta) = I(B) + \langle E-\beta \rangle \subseteq D,
\]
where $\< E-\beta\>$ corresponds to the Euler operators $E$ of the matrix $A$,
is called a \emph{Horn system with parameter $\beta$}.
\end{definition}

\begin{remark}
This is the binomial formulation for Horn systems. For the relation with the
classical systems of equations introduced by Appell and
Horn \cite{appell,horn89}, see \cite{dms, dmm}.
\end{remark}

In order to prove that Horn systems are Weyl closed, we need to describe their
solution spaces. This requires information about the lattice
basis ideal $I(B)$, namely, its primary decomposition. 
The main references for primary decomposition of
binomial ideals in general, and lattice basis ideals in particular, are
\cite{binomialideals, fischer-shapiro, hostenshapiro, dmm2}.

Each of the minimal primes of~$I(B)$ arises,
after row and column permutations, from a block decomposition of~$B$
of the form
\begin{equation}\label{eqn:MNOB}
  \left[
    \begin{array}{l|r}
    N & B_J\!\\\hline
    M & 0\ 
    \end{array}
  \right],
\end{equation}
where $M$ is a mixed submatrix of~$B$ of size $q \times p$ for some $0
\leq q \leq p \leq m$ \cite{hostenshapiro}.  (Matrices with $q = 0$
rows are automatically mixed; matrices with $q = 1$ row are never
mixed.)  We note that not all such decompositions correspond to 
minimal primes: the matrix $M$ has to satisfy another condition which
Ho\c{s}ten and Shapiro call irreducibility \cite[Definition~2.2 and
Theorem~2.5]{hostenshapiro}.  If $I(B)$ is a complete intersection,
then only square matrices $M$ will appear in the block
decompositions~(\ref{eqn:MNOB}), by a result of Fischer and Shapiro
\cite{fischer-shapiro}.

Let $\oJ$ be the set of indices of the $q$ rows of $M$ (before
permuting) and let $J = \{1,\dots,n \}\backslash \oJ$
be the index set of $B_J$ (again, before permuting).
Denote by $A_J$ the matrix whose columns are the columns of $A$ indexed by $J$.
Split the
variables $x_1,\ldots,x_n$ and $\del_1,\ldots,\del_n$ into two blocks
each:
\begin{align*}
  x_J = \{x_j : j \in J\}
  &\quad\text{and}\quad
  x_\oJ = \{x_j : j \notin J\}.
\\
  \del_J = \{\del_j : j \in J\}
  &\quad\text{and}\quad
  \del_\oJ = \{\del_j : j \notin J\}.
\end{align*}

Let $\sat(\ZZ B_J) = \QQ B_J \cap \ZZ^{J}$.
For each partial character $\rho : \sat(\ZZ B_J) \to \CC^*$ extending
the trivial character on~$\ZZ B_J$, the ideal 
\[
I_{\rho,J} = I_{\rho} + \langle \del_j : j \not \in J\rangle, \quad \mbox{where} \;
I_{\rho} = \langle \del_J^w - \rho(w-w') \del_J^{w'} : w,w' \in \NN^J, A_J w = A_J w'\rangle ,
\] 
is an
associated prime of~$I(B)$.
Note that the symbol $\rho$ here includes the specification 
of the sublattice
$\sat(\ZZ B_J) \subseteq \ZZ^J$.  

\begin{definition}(cf. \cite{dmm2}[Definition 4.3, Example 4.10])
If the matrix $M$ is square and invertible, the prime $I_{\rho,J}$ is
called a \emph{toral associated prime of $I(B)$}. The corresponding
primary component of $I(B)$, denoted by $C_{\rho,J}$, is called a 
\emph{toral component of $I(B)$}. Associated primes and primary
components that are not toral are called \emph{Andean}.
\end{definition}

The primary decomposition of $I(B)$, and in particular, its toral components,
are important to the study of Horn
systems because of the following fact \cite{dmm}[Theorem 6.8].

\[
\frac{D}{H(B,\beta)} \cong \bigoplus_{C_{\rho,J} \;\tiny{\rm toral}} 
\frac{D}{C_{\rho,J} + \langle E- \beta \rangle} \quad \mbox{for generic} \; \beta.
\]

This implies that, for generic $\beta$, the solution space of $H(B,\beta)$
is the direct sum of the solution spaces of the systems
$C_{\rho,J}+\langle E- \beta\rangle$, for toral $C_{\rho,J}$. In order to describe
these solution spaces, we need an explicit expression for the toral components
$C_{\rho,J}$.

Fix a toral component $C_{\rho,J}$ coming from a decomposition
(\ref{eqn:MNOB}).

Define a graph $\Gamma$ whose vertices are the points in $\NN^{\oJ}$.
Two vertices $u, u'  \in \NN^{\oJ}$  are connected by an edge
if $u-u'$ or $u'-u$ is a column of the matrix $M$. The connected
components of the graph $\Gamma$ are called the \emph{$M$-subgraphs of $\NN^{\oJ}$}.
If $u \in \NN^{\oJ}$, call $\Gamma(u)$ the $M$-subgraph that $u$ belongs to.
Then, by \cite{dmm2}[Corollary~4.14],
\[
C_{\rho,J} = I(B) + I_{\rho, J} + \langle \del_{\oJ}^u : \Gamma(u) \; \mbox{is unbounded}\rangle.
\]

Let $S$ be a set of representatives of the bounded $M$-subgraphs of $\NN^{\oJ}$.
By \cite{dmm}[Proposition 7.6], a basis of the space of polynomial solutions 
of the lattice basis ideal $I(M) \subseteq \CC[\del_{\oJ}]$ considered as a 
system of differential equations, consists of polynomials
\[
G_u = x^{u} \sum_{u+Mv \in \Gamma(u)} c_{v} x_{\oJ}^{Mv},  \quad u \in S,
\]
where the all the coefficients $c_v$ are nonzero.
Fix a basis $\cB_u$ of germs of holomorphic solutions
of
$I_{\rho,J} + \langle E - (\beta - A_{\oJ} u)\rangle$
at a generic nonsingular point,
where $A_{\oJ}$ is the matrix whose columns are the columns of $A$ indexed by 
$\oJ$.

By \cite{dmm}[Theorem 7.13], if $\beta$ is very generic,
a basis of the space 
of germs of holomorphic solutions of  $C_{\rho,J}+\langle E-\beta\rangle$  
at a generic nonsingular point
is given by the functions

\begin{equation}
\label{eqn:basis}
F_{u,f} = x_{\oJ}^u \sum_{u + Mv \in \Gamma(u)} c_v x_{\oJ}^{Mv} \del_J^{-Nv}(f) \; : \;
u \in S, \; f \in {\cB_u}.
\end{equation}

To make sense of the notation $\del_{J}^{-Nv}(f)$, we need the following result.

\begin{lemma}{\cite[Lemma 7.10]{dmm}}
\label{lemma:delJ}
If $\beta$ is very generic and $\alpha \in \NN^{J}$, the map
\[
\frac{D}{I_{\rho,J} + \langle E-\beta\rangle} \longrightarrow 
\frac{D}{I_{\rho,J} + \langle E- (\beta + A_J\alpha) \rangle} 
\]
given by right multiplication by $\del_J^{\alpha}$ is an isomorphism.
Consequently, if $P\del_{J}^{\alpha}$ belongs to 
$I_{\rho,J} + \langle E- (\beta + A_J\alpha) \rangle$,
then $P \in I_{\rho,J} + \langle E- \beta \rangle$.
\end{lemma}

The isomorphism from Lemma~\ref{lemma:delJ} implies that differentiation
$\del_J^{\alpha}$ is an isomorphism between the solution space of 
$I_{\rho,J}+\langle E- (\beta+A_{J}\alpha)\rangle$ 
and the solution space of $I_{\rho,J}+ \langle E-\beta\rangle$, whose
inverse we denote by $\del_J^{-\alpha}$. This
explains the notation $\del_J^{-Nv}f$ from (\ref{eqn:basis}).

\begin{theorem}
\label{thm:toral-weyl-closed}
Let $C_{\rho,J}$ be a toral component of a lattice basis ideal $I(B)$. If $\beta$
is very generic, then $C_{\rho,J}+\langle E-\beta\rangle$ is the annihilator
of its solution space, and is therefore Weyl closed.
\end{theorem}

\begin{proof}[Proof of Theorem~\ref{thm:toral-weyl-closed}]

First note that if $\beta$ is very generic, then 
$I_{\rho,J} + \langle E-\beta \rangle$
is Weyl closed. In fact, we may assume that this system has
a solution that can be represented as a fully supported power 
series on a translate of $\ker_{\ZZ}(A_J)$, and then
$I_{\rho,J} + \langle E-\beta \rangle$ is the differential annihilator 
of this function.

If $I_{\rho} = I_{A_J}$, this follows from Theorems~\ref{thm:good-sol}
and~\ref{thm:A-hyp-weyl-closed}.
To adapt those results to more general $I_{\rho}$, we note
that $I_{\rho}$ is isomorphic to $I_{A_J}$ by adequately rescaling the
variables $\del_j$, $j \in J$.

If $q = \# \oJ=0$, 
then the preceding paragraphs prove Theorem~\ref{thm:toral-weyl-closed},
so assume $q \neq 0$.

Pick a basis of germs of holomorphic solutions of 
$C_{\rho,J} + \langle E-\beta\rangle$
at a generic nonsingular point 
as in (\ref{eqn:basis}).
We assume that $\beta$ is generic enough that at least one
element of $\cB_u$ can be represented as a fully supported 
series on a translate of $\ker_{\ZZ}(A_J)$.

Let $P \in D$ that annihilates all the functions (\ref{eqn:basis}). We want to show that
$P \in C_{\rho,J}+\langle E-\beta\rangle$.

Write $P = \sum  \lambda_{\mu,\omu,\nu,\onu} x_J^{\mu} x_{\oJ}^{\omu} \del_{J}^{\nu} \del_{\oJ}^{\onu}$,
where all the $\lambda$s are nonzero complex numbers. 
We may assume that all
the $\onu$ appearing in $P$ belong to bounded $M$-subgraphs.

We introduce a partial order on the set of bounded $M$-subgraphs as
follows: $\Gamma(u) \leq \Gamma(u')$ if and only if there exist
elements $v \in \Gamma(u)$ and $v' \in \Gamma(u')$ such that
$v \leq v'$ coordinate-wise. Note that if $\Gamma(u) \leq \Gamma(u')$,
then for every $v \in \Gamma(u)$ there exists $v'\in \Gamma(u')$ such
that $v \leq v'$.

Consider the set $\{ \Gamma(\onu) \}$ of bounded $M$-subgraphs which 
have representatives in $P$, and choose a minimal element in this set,
$\Gamma(\ogamma)$, and a corresponding term in $P$, 
$\lambda_{\alpha,\oalpha,\gamma,\ogamma} x_J^{\alpha} x_{\oJ}^{\oalpha} \del_J^{\gamma} \del_{\oJ}^{\ogamma}$.

Now $\ogamma \in \Gamma(u)$ for some $u \in S$. Consider one of the functions
$F_{u,f}$ from (\ref{eqn:basis}). We know that $P F_{u,f} = 0$. 
Also,
\begin{equation}
\label{eqn:term-to-cancel}
\lambda_{\alpha,\oalpha,\gamma,\ogamma} x_J^{\alpha} x_{\oJ}^{\oalpha} \del_J^{\gamma} \del_{\oJ}^{\ogamma}
F_{u,f} = \ogamma! \lambda_{\alpha,\oalpha,\gamma,\ogamma} x_J^{\alpha} x_{\oJ}^{\oalpha}
\del_{J}^{\gamma-Nz}f(x_J),
\end{equation}
where $\ogamma!$ is the product of the factorials of the coordinates of $\gamma$,
and $\ogamma = u + Mz$. 
The reason that only one term of $F_{u,f}$ survives is that
the $x_{\oJ}$ monomials appearing in $F_{u,f}$ are of the form $\ogamma + My$. If
we could find $My$ whose coordinates are all positive, then there would be 
no bounded $M$-subgraphs. Therefore, some coordinate of $My$ must be negative, so that
applying $\del_{\oJ}^{\ogamma}$ to $x_{\oJ}^{\ogamma+My}$ gives zero.

In order to cancel the term from (\ref{eqn:term-to-cancel}), $P$
must contain a term 
$\lambda_{\sigma,\osigma,\tau,\otau} x_J^{\sigma} x_{\oJ}^{\osigma} \del_{J}^{\tau} \del_{\oJ}^{\otau}$
such that $\del_{\oJ}^{\otau}$ does not kill all the monomials 
$x_{\oJ}^{v} : v \in \Gamma(u)=\Gamma(\ogamma)$. If $\otau \not\in \Gamma(\ogamma)$,
we would have that $\Gamma(\otau) < \Gamma(\ogamma)$, which contradicts the
choice of $\ogamma$. Thus, $\otau = \ogamma + My$ for some $y$.

Now, that
$\lambda_{\sigma,\osigma,\tau,\otau} x_J^{\sigma} x_{\oJ}^{\osigma} \del_{J}^{\tau} \del_{\oJ}^{\otau} F_{u,f}$
is a multiple of (\ref{eqn:term-to-cancel}), 
means that
\begin{equation}
\label{eqn:multiple}
\ogamma! \lambda_{\alpha,\oalpha,\gamma,\ogamma} x_J^{\alpha} x_{\oJ}^{\oalpha}
\del_{J}^{\gamma}\del_J^{-Nz}f(x_J)
= c \otau! \lambda_{\sigma,\osigma,\tau,\otau} x_{J}^{\sigma} x_{\oJ}^{\osigma} \del_J^{\tau} 
\del_J^{-N(z+y)}f(x_J)
\end{equation}
for some nonzero $c$.
Therefore $\osigma=\oalpha$.

Assume that $\tau - Ny$ is coordinate-wise non-negative. (The case when
$\tau - Ny$ has some strictly negative coordinates is resolved by
multiplying $P$ on the right by a suitable monomial in the variables 
$\del_J$, working with a different (albeit still very generic) parameter, and
then applying Lemma~\ref{lemma:delJ}).

Then formula (\ref{eqn:multiple}) implies that
\[
\bigg(
\ogamma! \lambda_{\alpha,\oalpha,\gamma,\ogamma} x_J^{\alpha} \del_{J}^{\gamma} -
c  \otau! \lambda_{\sigma,\osigma,\tau,\otau} x_{J}^{\sigma}\del_J^{\tau-Ny} 
\bigg) \del_J^{-Nz} f(x_J) = 0.
\]
Since $\del_{J}^{-Nz} f(x_J)$ is a (fully supported) solution of 
$I_{\rho,J} + \langle E - (\beta- A_{\oJ}u + A_JNz)\rangle$,
we conclude that
\[
\big(
\ogamma! \lambda_{\alpha,\oalpha,\gamma,\ogamma} x_J^{\alpha} \del_{J}^{\gamma} -
c  \otau! \lambda_{\sigma,\osigma,\tau,\otau} x_{J}^{\sigma}\del_J^{\tau-Ny} 
\big) \in  I_{\rho,J} +
\langle E - (\beta -A_{\oJ}u +A_JNz ) \rangle.
\]
Write 
\[
\big(
\ogamma! \lambda_{\alpha,\oalpha,\gamma,\ogamma} x_J^{\alpha} \del_{J}^{\gamma} -
c  \otau! \lambda_{\sigma,\osigma,\tau,\otau} x_{J}^{\sigma}\del_J^{\tau-Ny} 
\big)
=
Q + \sum_{j \in \oJ} Q_{j} \del_{\oJ}^{e_j},
\]
where $Q \in I_{\rho} + \langle E-(\beta-A_{\oJ}u +A_J Nz )\rangle$.
Then
\[
\big(
\ogamma! \lambda_{\alpha,\oalpha,\gamma,\ogamma} x_J^{\alpha} \del_{J}^{\gamma} -
c  \otau! \lambda_{\sigma,\osigma,\tau,\otau} x_{J}^{\sigma}\del_J^{\tau-Ny} 
\big) 
\del_{\oJ}^{\ogamma} = Q \del_{\oJ}^{\ogamma} + \sum_{j \in \oJ}\del_{\oJ}^{\ogamma+e_j}, \]
where
$Q \del_{\oJ}^{\ogamma}\in I_{\rho} +\langle E - (\beta -A_{\oJ}u +A_JNz + A_{\oJ} \ogamma)  \rangle$.
As
\[
\beta -A_{\oJ}u +A_JNz + A_{\oJ} \ogamma =\beta -A_{\oJ}u + A_JNz  + A_{\oJ} (u+Mz) =
\beta + A(Nz+Mz) = \beta,
\]
we have $Q \del_{\oJ}^{\ogamma} \in I_{\rho} + \langle E -\beta \rangle$.

Define
\[
\begin{array}{l}
P_o = \\
=\lambda_{\alpha,\oalpha,\gamma,\ogamma} x_J^{\alpha} x_{\oJ}^{\oalpha} \del_J^{\gamma} \del_{\oJ}^{\ogamma}
+\lambda_{\sigma,\osigma,\tau,\otau} x_J^{\sigma} x_{\oJ}^{\osigma} \del_{J}^{\tau} \del_{\oJ}^{\otau}
- (\lambda_{\alpha,\oalpha,\gamma,\ogamma} c \frac{\otau!}{\ogamma!} + \lambda_{\sigma,\osigma,\tau,\otau}) 
x_J^{\sigma}x_{\oJ}^{\osigma} \del_J^{\tau-Ny} \del_{\oJ}^{\ogamma} \\
= 
x_{\oJ}^{\oalpha}\big(
\lambda_{\alpha,\oalpha,\gamma,\ogamma} x_J^{\alpha} \del_{J}^{\gamma} -
c  \frac{\otau!}{\ogamma!} \lambda_{\sigma,\osigma,\tau,\otau} x_{J}^{\sigma}\del_J^{\tau-Ny} 
\big)  \del_{\oJ}^{\ogamma}
+
\lambda_{\sigma,\osigma,\tau,\otau} x_J^{\sigma} x_{\oJ}^{\osigma} 
(\del_{J}^{\tau}\del_{\oJ}^{\otau} - \del_{J}^{\tau-Ny}\del_{\oJ}^{\otau-My}) \\

= x_{\oJ}^{\oalpha} Q \del_{\oJ}^{\ogamma} + \sum_{j \in \oJ} x_{\oJ}^{\oalpha} Q_j \del^{\ogamma+e_j}
+\lambda_{\sigma,\osigma,\tau,\otau} x_J^{\sigma} x_{\oJ}^{\osigma} 
(\del_{J}^{\tau}\del_{\oJ}^{\otau} - \del_{J}^{\tau-Ny}\del_{\oJ}^{\otau-My}) \\

\equiv \sum_{j \in \oJ} x_{\oJ}^{\oalpha} Q_{j} \del_{\oJ}^{\ogamma+e_j} \quad \mbox{mod}\, (C_{\rho,J} + 
\langle E- \beta \rangle),
\end{array}
\]
since 
$Q\del_{\oJ}^{\ogamma} \in I_{\rho} + \langle E-\beta \rangle$
and
$\del_{J}^{\tau}\del_{\oJ}^{\otau} - \del_{J}^{\tau-Ny}\del_{\oJ}^{\otau-My} \in I(B)$.

Now consider the operator $P-P_o -\sum_{j \in \oJ} x_{\oJ}^{\oalpha} Q_j \del_{\oJ}^{\ogamma+e_j}$,
which is congruent to $P$ modulo $C_{\rho,J} + \langle E-\beta \rangle$. 
Note that this eliminates two of the terms in $P$ 
at the cost of adding terms with strictly higher
monomials in $\del_{\oJ}$ than $\del_{\oJ}^{\ogamma}$, and possibly adding a multiple of
$x_J^{\sigma}x_{\oJ}^{\osigma} \del_J^{\tau-Ny} \del_{\oJ}^{\ogamma}$. 

We apply the same treatment to $P-P_o -\sum_{j \in \oJ} x_{\oJ}^{\oalpha} Q_j \del_{\oJ}^{\ogamma+e_j}$
that we did to $P$, and repeat. Eventually, this procedure will get rid of all the
terms that have $\del_{\oJ}^{u}$ with $\Gamma(u)$ bounded.

We conclude that $P \in C_{\rho,J} + \langle E - \beta \rangle$.
\end{proof}

We need one more ingredient to prove that Horn systems are Weyl closed.

\begin{proposition}
\label{prop:int}
If $\beta$ is generic, and $C_{\rho_1,J_1}, \dots,C_{\rho_r,J_r}$ are the toral components
of the lattice basis ideal $I(B)$, then
\begin{equation}
\label{eqn:intersection}
\bigcap_{i=1}^r \big(C_{\rho_i,J_i} + \langle E-\beta \rangle \big) 
= I(B) +\langle E-\beta \rangle = H(B,\beta). 
\end{equation}
\end{proposition}

\begin{proof}
The inclusion $\supseteq$ follows from the fact that
$\cap_{i=1}^r C_{\rho_i,J_i} \supseteq I(B)$. 
To prove the reverse inclusion, let
$P$ be an element of the left hand side of (\ref{eqn:intersection}).

By the proof of \cite{dmm}[Theorem 6.8], the natural map
\[
\frac{D}{H(B,\beta)} 
\longrightarrow 
\bigoplus_{i=1}^r \frac{D}{C_{\rho_i,J_i} + \langle E-\beta \rangle}
\]
is an isomorphism when $\beta$ is generic. 
Since $P$ belongs to the left hand side of
(\ref{eqn:intersection}), its image under this map is zero. Therefore
$P$ must be an element of $H(B,\beta)$.
\end{proof}

\begin{corollary}
\label{coro:horn-weyl-closed}
For very generic $\beta$, the Horn system $H(B,\beta)$ is Weyl closed.
\end{corollary}

\begin{proof}
For generic $\beta$, Proposition~\ref{prop:int} says that
$H(B,\beta)$ is the intersection of
the systems $C_{\rho,J} + \langle E - \beta \rangle$ corresponding
to the toral components of $I(B)$. Each of these is Weyl closed 
for very generic parameters
by Theorem~\ref{thm:toral-weyl-closed}. 
We finish by noting that the intersection of Weyl closed $D$-ideals
is Weyl closed.
\end{proof}


\raggedbottom
\def\cprime{$'$} \def\cprime{$'$}
\providecommand{\MR}{\relax\ifhmode\unskip\space\fi MR }
\providecommand{\MRhref}[2]{%
  \href{http://www.ams.org/mathscinet-getitem?mr=#1}{#2}
}
\providecommand{\href}[2]{#2}

\end{document}